\documentclass[12pt]{amsart}

\newtheorem{theorem}{Theorem}[section]
\newtheorem{lemma}[theorem]{Lemma}
\newtheorem{prop}[theorem]{Proposition}

\theoremstyle{definition}

\theoremstyle{remark}
\newtheorem{remark}[theorem]{Remark}

\newcommand{\mysection}[1]{\section{#1}
\setcounter{equation}{0}}

\newcommand{\eqdef}{\overset{\mbox{\tiny{def}}}{=}}

\newcommand{\bR}{\mathbb R}

\newcommand{\Div}{\operatorname{div}}

\renewcommand{\epsilon}{\varepsilon}

\newcommand{\alK}{\alpha_k}

\begin{document}
\title[Partial regularity of 6D steady solutions to NSE] {On partial regularity of
steady-state solutions to the 6D Navier-Stokes equations}

\author[H. Dong]{Hongjie Dong}
\address[H. Dong]{Division of Applied Mathematics, Brown University,
182 George Street, Providence, RI 02912, USA}
\email{Hongjie\_Dong@brown.edu}
\thanks{H. Dong was partially supported by the NSF grant DMS-0800129.}

\author[R. M. Strain]{Robert M. Strain}
\address[R. M. Strain]{Department of Mathematics, University of Pennsylvania
David Rittenhouse Lab., 209 South 33rd Street, Philadelphia, PA 19104, USA}
\email{strain at math.upenn.edu}
\thanks{R. M. Strain was partially supported by the NSF grant DMS-0901463.}




\begin{abstract}
Consider steady-state weak solutions to the incompressible
Navier-Stokes equations in six spatial dimensions. We prove that the 2D
Hausdorff measure of the set of singular points is equal to zero.  This problem was mentioned in 1988 by Struwe
\cite{Str88}, during his study of the five dimensional case.
\end{abstract}

\setcounter{tocdepth}{1}
\maketitle

\thispagestyle{empty}

\mysection{Introduction}

In this paper we consider the incompressible steady-state Navier-Stokes
equations in {\em six} spatial dimensions with unit viscosity
\begin{equation}
                            \label{NSeq1}
u\nabla u-\Delta u+\nabla p=f,\quad \text{div}\, u=0,
\end{equation}
in a domain $\Omega\subset \bR^6$. We are interested in the partial
regularity of weak solutions $(u,p)$ to \eqref{NSeq1}.

Although the problem of the global regularity of solutions to
the time-dependent Navier-Stokes equations in three and higher space dimensions
is still widely open,
many authors have studied the partial regularity of weak solutions.
In his pioneering work \cite{Sch1, Sch2, Sch4}, Scheffer established various partial regularity results for weak solutions to the 3D Navier-Stokes equations satisfying the so-called local energy inequality.
In 3D, the notion of {\em suitable weak solutions} was first introduced in a celebrated paper \cite{CKN} by Caffarelli,
Kohn and Nirenberg. They called a pair consisting of velocity $u$
and pressure $p$ a suitable weak solution if $u$ has finite energy
norm, $p$ belongs to the Lebesgue space $L_{5/4}$, $u$ and $p$ are
weak solutions to the Navier-Stokes equations, and they satisfy a local
energy inequality. After proving an $\epsilon$-regularity criteria
for local boundedness of solutions, they established partial regularity of solutions and estimated the Hausdorff dimension of the singular set. They proved that, for any suitable weak solution $(u, p)$, there is an open subset
where the velocity field $u$ is regular and they showed
that the 1D Hausdorff measure of the complement of this subset is
equal to zero. In \cite{flin}, F. Lin gave
a more direct and simplified proof of Caffarelli, Kohn and Nirenberg's
result. A detailed treatment was later given by Ladyzhenskaya and
Seregin in \cite{OL2}. Recently, some extended results
have been obtained in a number of papers; see Seregin \cite{Seregin}, Gustafson, Kang and Tsai \cite{tsai}, and Vasseur \cite{Va07}, Kukavica \cite{Ku08}, and the references therein.

Much fewer results are available in the literature for the 4D and higher dimensional time-dependent Navier-Stokes equations, in which case the problem is more super-critical. In \cite{Sch3}, Scheffer showed that there exists a
weak solution $u$ in  $\bR^4\times \bR^+$, which may not necessarily verify the local
energy estimate, such that $u$ is
continuous outside a locally closed set of $\bR^4\times \bR^+$ whose
3D Hausdorff measure is finite. In a recent paper \cite{DD07}, the first author and D. Du proved that, for any local-in-time smooth solution to the 4D Navier-Stokes equations, the 2D Hausdorff measure of the set of singular points at the first potential blow-up
time is equal to zero. We remark that, in terms of the method, the dimension four in \cite{DD07} is critical due to the following reason. To the best of our knowledge all the existing methods on partial regularity for the Navier-Stokes equations share the following
prerequisite condition: in the energy inequality the nonlinear term
should be controlled by the energy norm under the Sobolev imbedding
theorem. Actually, four is the highest dimension in which we have
such condition: $L_3^tL_3^x\hookrightarrow L_\infty^tL_2^x\cap L_2^tH^1$. In five or higher dimensions this condition fails.

This paper concerns the partial regularity of weak solutions $u$ to the steady-state Navier-Stokes equations \eqref{NSeq1}. In the literature, the most relevant paper on the subject is \cite{Str88} by Struwe, in which he proved the following $\epsilon$-regularity result, which implies that weak solutions are regular outside a singular set of zero 1D Hausdorff measure.

\begin{theorem}[Struwe \cite{Str88}]
                    \label{thm0}
Let $\Omega$ be an open domain in $\bR^5$ and $f\in L_q(\Omega)$ for some $q>5/2$. There exists an absolute constant $\epsilon_0>0$ such that the following holds true. If $u\in H^1(\Omega;\bR^5)$ is a weak solution to \eqref{NSeq1} which satisfies a generalized energy inequality, and if for some $x_0\in \Omega$ there is $R_0>0$ such that
$$
r^{-1}\int_{|x-x_0|<r}|\nabla u|^2\,dx\le \epsilon_0, \quad \forall r\in (0,R_0),
$$
then $u$ is H\"older continuous in a neighborhood of $x_0$.
\end{theorem}

The proof of Theorem \ref{thm0} relies on some techniques for proving regularity for elliptic systems (cf. Morrey \cite{Mo66} and Giaquinta \cite{Gi83}) and estimates for the linear Stokes' system due to Solonnikov \cite{So68}. A similar approach was used before by Giaquinta and Modica in \cite{GM82} to study the steady-state Navier-Stokes equations in dimensions $d\le 4$. Because time corresponds to two space dimensions, in some sense the 5D stationary problem is similar to the 3D non-stationary problem. Therefore, dimension five is the smallest dimension for which the steady-state Navier-Stokes equations are super critical. We also note that Theorem \ref{thm0} was improved by K. Kang \cite{Ka04}, in which partial regularity up to the boundary was established for a smooth domain $\Omega\subset \bR^5$. The existence of regular solutions to the steady-state Navier-Stokes in high dimensions have also attracted substantial attention. We refer the reader to \cite{Ge79,Str95,FR94,FR96,FR95,FR98,FS09} and the references therein.

In \cite{Str88} Struwe raised the following interesting question: does the analogous partial regularity result hold in dimension six or higher.  It seems to us that some arguments in \cite{Str88} do not work in six or higher dimensions.
In this paper, we give a positive answer to Struwe's question in dimension six. To be more precise, we shall prove the following regularity result with a sufficiently small constant $\epsilon_0$: Let $\Omega$ be an open set in $\bR^6$, $f\in L_{6,\text{loc}}(\Omega)$, and let $u$ be a weak solution to \eqref{NSeq1} satisfying a local energy inequality \eqref{energy}. Then if for some $x_0\in \Omega$ there exists an $R_0>0$ such that
$$
r^{-2}\int_{|x-x_0|<r}|\nabla u|^2\,dx\le \epsilon_0,\quad \forall r\in (0,R_0),
$$
then $u$ is H\"older continuous in a neighborhood of $x_0$. In particular, it follows that the 2D Hausdorff measure of the set of singular points of the weak solution $u$ is equal to zero.

Related to Struwe's work, our proof also uses some techniques which appeared in the study of the regularity theory for elliptic systems; but our approach is quite different from that in \cite{Str88}. In particular, we do not use any estimate for the linear Stokes' systems. Roughly speaking, there are three steps in our proof.  In the first step, we essentially follow the argument in \cite{DD07}, which in turn used some ideas in \cite{OL2} and \cite{flin}. The novelties are in the second and third steps. In the second step, we choose suitable test functions in the energy inequality and then use an iteration method to establish a weak decay estimate of certain scaling invariant quantities. In the last step, we successively improve this decay estimate by a bootstrap argument, and finally we use the elliptic theory to get a good estimate of the $L_{3/2}$ norm of $\nabla u$, which yields the H\"older regularity thanks to the Morrey lemma.

It is, however, unclear to us whether similar results hold in seven or higher dimensions. In fact, we believe that six is the highest
dimension to which our approach can be applied; see Remark \ref{finalrem}. Therefore, in some sense, our results are critical in terms of the dimension.

To conclude this introduction, we explain some notation used in the sequel:
$\bR^d$ is the $d$-dimensional Euclidean space with a fixed
orthonormal basis. A typical point in ${\mathbb R}^d$ is denoted by
$x=(x_1,x_2,...,x_d)$. As usual the summation convention over
repeated indices is enforced. And $x\cdot y=x_jy_j = \sum_{j=1}^d x_j y_j$ is the inner
product for $x,y\,\in\bR^d$.
The standard Lebesgue spaces are denoted by $L_p$ $(p\ge 1)$.
Various constants
are denoted by $N$ in general and the expression $N=N(\cdots)$ means
that the given constant $N$ depends only on the contents of the
parentheses.

\mysection{Setting and main results}
                                        \label{Sec2}

For summable functions $p,u=(u_i)$ and $\tau=(\tau_{ij})$, we use
the following standard differential operators
$$
u_{,i}=D_i u,\quad \nabla
p=(p_{,i}),\quad \nabla u=(u_{i,j}),
$$
$$
\text{div}\,u=u_{i,i},\quad\Div\tau=(\tau_{ij,j}),\quad
\Delta u=\text{div}\nabla u.
$$
These are all understood in the sense of distributions.

We use the following notation for spheres and balls
$$
S(x_0,r)=\{x\in \bR^6\vert \,|x-x_0|=r\},\quad S(r)=S(0,r),\quad
S=S(1),
$$
$$
B(x_0,r)=\{x\in \bR^6\vert \,|x-x_0|<r\},\quad B(r)=B(0,r),\quad B=B(1).
$$
We also denote the mean value of a summable function as follows
$$
[u]_{x_0,r}=\frac{1}{|B(r)|}\int_{B(x_0,r)}u(x)\,dx.
$$
Here $|A|$ as usual denotes the Lebesgue measure of the set $A$.

Let $x_0$ be a given point in $\Omega$ and $r>0$ a real
number such that $B(x_0,r)\subset \Omega$. It is known that
in the sense of distributions one has
\begin{multline*}
\Delta p=-D_{ij}\big(u_iu_j\big)+\Div f
\\
=-D_{ij}\big((u_i-[u_i]_{x_0,r})(u_j-[u_j]_{x_0,r})\big)+\Div f, \quad
\text{in}\,\,B(x_0,r).
\end{multline*}
This will hold for a weak solution to \eqref{NSeq1}.
Now let $\eta(x)$ be a smooth function on $\bR^6$ supported in the unit
ball $B(1)$, $0\leq \eta\leq 1$ and $\eta\equiv 1$ on $\bar B(2/3)$.
We consider the decomposition
\begin{equation}\label{pressDECOMP}
p={\tilde
p}_{x_0,r}+h_{x_0,r},
\quad
\text{in} ~ B(x_0,r).
\end{equation}
Above $\tilde p_{x_0,r}$ solves the following Poisson equation
$$
\Delta \tilde p_{x_0,r}
=
-D_{ij}\big((u_i-[u_i]_{x_0,r})(u_j-[u_j]_{x_0,r})\eta_{x_0,r}\big)
+\Div(f\eta_{x_0,r}),
$$
where $\eta_{x_0,r}(x)\eqdef \eta((x-x_0)/r)$.
Then $h_{x_0,r}$
is harmonic in $B(x_0,2r/3)$.

We will omit the indices of $\tilde p$ and $h$ whenever there is no possibility of confusion. The following
notation will be used throughout the article:
\begin{align*}
A(r)&=A(r,x_0)=\frac{1}{r^4}\int_{B(x_0,r)}|u|^2\,dx,\\
E(r)&=E(r,x_0)=\frac{1}{r^2}\int_{B(x_0,r)}|\nabla u|^2\,dx,\\
C(r)&=C(r,x_0)=\frac{1}{r^3}\int_{B(x_0,r)}|u|^3\,dx,\\
D(r)&=D(r,x_0)=\frac{1}{r^3}\int_{B(x_0,r)}|p-[h]_{x_0,r}|^{3/2}\,dx,\\
F(r)&=F(r,x_0)=\int_{B(x_0,r)}|f|^{2}\,dx.
\end{align*}
Notice these objects are invariant under the natural scaling for \eqref{NSeq1}:
$$
u(x)\to \lambda u(\lambda x),\quad p(x)\to \lambda^2 p(\lambda x),\quad
f(x)\to \lambda^3 f(\lambda x).
$$
We will use these quantities to study the regularity of 6D steady-state suitable weak solutions to \eqref{NSeq1}.

We say that a pair of functions $(u,p)$ is a suitable weak solution to \eqref{NSeq1} in $\Omega$ if $u\in H^1_{\text{loc}}(\Omega)$ and $p\in L_{3/2,\text{loc}}(\Omega)$ satisfy \eqref{NSeq1} in the weak sense and additionally the generalized local energy inequality holds for any non-negative test function $\psi\in C_0^\infty(\Omega)$:
\begin{equation}
                    \label{energy}
2\int_{\Omega}|\nabla u|^2\psi\,dx
\leq \int_{\Omega}|u|^2\Delta \psi+(|u|^2+2p)u\cdot
\nabla\psi+f\cdot u ~ \psi\,dx.
\end{equation}
The existence of regular solutions to the Dirichlet problem of \eqref{NSeq1} in four dimensions have been obtained Gerhardt \cite{Ge79}, in five dimensions by Struwe \cite{Str95}, and in five and six dimensions by Frehse and Ruzicka  \cite{FR94,FR96}; see also \cite{heywood,FR95,FR98,FS09} for other related results. We observe that the regular solution constructed in \cite{FR96} satisfies \eqref{energy}. On the other hand,
the uniqueness does not hold in general unless some smallness assumption is imposed on the data.

By the Sobolev imbedding theorem,
for any functions $u\in H^1_{\text{loc}}$ and $r>0$, we have the following inequality
\begin{equation}
                        \label{eq11.30}
\int_{B(r)}|u|^3\,dx\leq N
\left(\int_{B(r)}|\nabla u|^2\, dx\right)^{3/2}
 +N r^{-3}\left(\int_{B(r)}|u|^2\,dx\right)^{3/2}.
\end{equation}
This implies that $C(r)$ is well defined for a suitable weak solution.

Next we state the main results of the article.

\begin{theorem}[$\epsilon$-regularity criterion in terms of $E$]
                    \label{thm3}
Let $\Omega$ be an open set in $\bR^6$, $f\in L_{6,\text{loc}}(\Omega)$, and suppose that the pair
$(u,p)$ is a suitable weak
solution to \eqref{NSeq1} in $\Omega$. Then there is a positive number
$\epsilon_0$ satisfying the following property. Assume that for a
point $x_0\in \Omega$  the following inequality holds:
\begin{equation}
                                    \label{eq4.1.12}
\limsup_{r\downarrow 0}E(r)\leq \epsilon_0.
\end{equation}
Then $u$ is H\"older continuous in a neighborhood of $x_0$.
\end{theorem}

\begin{theorem}[$\epsilon$-regularity criterion in terms of $C$, $D$ and $F$]
                    \label{thm4}
Let $\Omega$ be an open set in $\bR^6$, $f\in L_{6,\text{loc}}(\Omega)$, and suppose that the pair
$(u,p)$ is a suitable weak
solution to \eqref{NSeq1} in $\Omega$. There is a positive number
$\epsilon_0$ satisfying the following property. Assume that for a
point $x_0\in \Omega$ and for some $\rho_0$ such that $B(x_0,\rho_0)\subset \Omega$
the inequality
\begin{equation}\label{eq4.1.11}
C(\rho_0)+D(\rho_0)+F(\rho_0)\leq \epsilon_0
\end{equation}
holds. Then $u$ is H\"older continuous in a neighborhood of $x_0$.
\end{theorem}

\begin{theorem}[Partial regularity]
                    \label{thm1}
Let $\Omega$ be an open set in $\bR^6$, $f\in L_{6,\text{loc}}(\Omega)$, and suppose that the pair
$(u,p)$ is a suitable weak
solution to \eqref{NSeq1} in $\Omega$. Then the 2D Hausdorff
measure of the set of singular points in $\Omega$ is equal
to zero.
\end{theorem}

These results are in the spirit of the work of Caffarelli,
Kohn and Nirenberg \cite{CKN}, where it was proved that for any suitable weak solution $u$ to the 3D time-dependent Navier-Stokes equations the 1D Hausdorff measure of the set of singular points is equal to zero. For 5D steady-state Navier-Stokes equations, results of this type were obtained by Struwe \cite{Str88}.

\begin{remark}
The assumption of the external force $f$ in Theorems \ref{thm3}-\ref{thm1} may be relaxed. It should be possible to only assume that $f$ in $L_{p,\text{loc}}$ for some $p\in (3,6)$ or alternatively to assume that $f$ is in certain Morrey spaces. However, we do not intend to find such a minimal assumption of $f$ in this paper.
\end{remark}

\mysection{The proof}

We shall prove the main theorems in three steps.

\subsection{Step 1} In the first step, we want to
control $A$, $C$ and $D$ in a smaller ball by the their values in a larger
ball under the assumption that $E$ is sufficiently small. Here we follow the argument in \cite{DD07}, which in turn used some ideas in \cite{OL2} and \cite{flin}.  These first few estimates do not use the equation \eqref{NSeq1}.  Furthermore, in this section we often write $C(r) = C(r, x_0)$ and similarly for $A$, $D$, $E$ and $F$.

\begin{lemma}
                    \label{lemma2.28.1}
Suppose $\gamma\in (0,1)$, $\rho>0$ are constants and
$B(x_0,\rho)\subset \Omega$. Then we have
\begin{equation}\label{eq11.35}
C(\gamma\rho)\leq N\big[\gamma^{-3}E^{3/2}(\rho)
+\gamma^{-6}A^{3/4}(\rho)E^{3/4}(\rho) +\gamma^3 C(\rho)\big],
\end{equation}
where $N$ is a constant independent of $\gamma$, $\rho$ and $x_0$.
\end{lemma}

\begin{proof}
Denote $r=\gamma \rho$. By using the Poincar{\'e} inequality
 and Cauchy's inequality,
we have
\begin{multline*}
\int_{B(x_0,r)}|u|^2\,dx=\int_{B(x_0,r)}
\big(|u|^2-[|u|^2]_{x_0,\rho}\big)\,dx
+\int_{B(x_0,r)}[|u|^2]_{x_0,\rho}\,dx
\\
\leq N\rho\int_{B(x_0,\rho)}|\nabla u||u|\,dx
+\left(\frac{r}{\rho}\right)^6\int_{B(x_0,\rho)}|u|^2\,dx.
\end{multline*}
This is furthermore bounded by
\begin{multline*}
\leq N\rho\Big(\int_{B(x_0,\rho)}|\nabla u|^2\,dx\Big)^{1/2}
\Big(\int_{B(x_0,\rho)}|u|^2\,dx\Big)^{1/2}
\\
+\left(\frac{r}{\rho}\right)^6\int_{B(x_0,\rho)}|u|^2\,dx
\\
\leq N\rho^3A^{1/2}(\rho) \Big(\int_{B(x_0,\rho)}|\nabla
u|^2\,dx\Big)^{1/2} +\left(\frac{r}{\rho}\right)^6
\Big(\int_{B(x_0,\rho)}|u|^3\,dx\Big)^{2/3}\rho^{2}.
\end{multline*}
Due to the Sobolev inequality \eqref{eq11.30}, we obtain
\begin{multline*}
\int_{B(x_0,r)}|u|^3\,dx
\leq
N
\left[\left(\int_{B(x_0,r)}|\nabla u|^2\,dx\right)^{3/2}  \right.
\\
\left.
+
r^{-3} \rho^{9/2}A^{3/4}(\rho)\left(\int_{B(x_0,\rho)}|\nabla u|^2\,dx\right)^{3/4}
+\left(\frac{r}{\rho}\right)^6\int_{B(x_0,\rho)}|u|^3\,dx
\right].
\end{multline*}
The conclusion of Lemma \ref{lemma2.28.1} follows immediately.
\end{proof}

\begin{lemma}
                    \label{lemma2.28.2}
Suppose $\gamma\in (0,1/4]$ and $\rho>0$ are
constants, and $B(x_0,\rho)\subset \Omega$. Then for any $x_1\in B(x_0,\rho/4)$ we have
\begin{equation}
                                \label{eq11.30d}
D(\gamma\rho,x_1)\leq N\big[\gamma^{9/2} D(\rho)
+\gamma^{-3}E^{3/2}(\rho)+\gamma^{-3}F^{3/4}(\rho)\big],
\end{equation}
where $N$ is a constant independent of $\gamma$, $\rho$, $x_0$, and $x_1$.
\end{lemma}
\begin{proof}
Denote $r=\gamma \rho$.
Recall the decomposition of $p$ introduced in \eqref{pressDECOMP}. By
using the Calder\'on-Zygmund estimate and the Sobolev-Poincar{\'e} inequality, one has
\begin{multline}                    \label{eq28.1}
\int_{B(x_1,r)}|{\tilde p}_{x_1,r}(x)|^{3/2}\,dx
\\
\leq
N\int_{B(x_1,r)} |u-[u]_{x_1,r}|^{3}\,dx+N\int_{\bR^6}|\Delta^{-1}\Div (f\eta_{x_1,r})|^{3/2}\,dx
\\
\leq N\Big(\int_{B(x_1,r)}|\nabla u|^2\,dx\Big)^{3/2}
+N\Big(\int_{B(x_1,r)}|f|^{6/5}\,dx\Big)^{5/4}.
\end{multline}
Similarly,
\begin{multline}
                            \label{eq065.23}
\int_{B(x_0,\rho)}|{\tilde p}_{x_0,\rho}|^{3/2}\,dx \leq
N\Big(\int_{B(x_0,\rho)}|\nabla u|^2\,dx\Big)^{3/2}
\\
+N\Big(\int_{B(x_0,\rho)}|f|^{6/5}\,dx\Big)^{5/4}.
\end{multline}
Since $h_{x_0,\rho}$ is harmonic in $B(x_0,2\rho/3)$, any Sobolev
norm of $h_{x_0,\rho}$ in a smaller ball can be estimated by any $L_p$ norm of $h_{x_0,\rho}$ in a larger ball.
Thus, using the Poincar{\'e}
inequality, one can obtain
\begin{multline} \notag
\int_{B(x_1,r)}|h_{x_0,\rho}-[h_{x_0,\rho}]_{x_1,r}|^{3/2}\,dx
\\
\leq Nr^{3/2}\int_{B(x_1,r)}|\nabla
h_{x_0,\rho}|^{3/2}\,dx \leq
Nr^{15/2}\sup_{B(x_1,r)}|\nabla h_{x_0,\rho}|^{3/2}.
\end{multline}
Further using the estimates for harmonic functions and the inclusion $B(x_1,r)\subset B(x_1,\rho/3)\subset B(x_0,2\rho/3)$, the above is
\begin{multline}
 \leq
N\left(\frac{r}{\rho}\right)^{15/2}
\int_{B(x_1,\rho/3)}|h_{x_0,\rho}-[h_{x_0,\rho}]_{x_0,\rho}|^{3/2}\,dx
\\
\label{eq28.2}
\leq N\left(\frac{r}{\rho}\right)^{15/2}
\int_{B(x_0,\rho)}|p-[h_{x_0,\rho}]_{x_0,\rho}|^{3/2}
+|{\tilde p}_{x_0,\rho}|^{3/2}\,dx.
\end{multline}
Combining \eqref{eq065.23} and \eqref{eq28.2} together yields
\begin{multline} \label{eq065.42}
\int_{B(x_1,r)}|p-[h_{x_0,\rho}]_{x_1,r}|^{3/2}\,dx
\leq N\Big(\int_{B(x_0,\rho)}|\nabla u|^2\,dx\Big)^{3/2}
\\
+N\Big(\int_{B(x_0,\rho)}|f|^{6/5}\,dx\Big)^{5/4}
+N\left(\frac{r}{\rho}\right)^{15/2}
\int_{B(x_0,\rho)}|p-[h_{x_0,\rho}]_{x_0,\rho}|^{3/2}\,dx.
\end{multline}
Since $\tilde p_{x_1,r}+h_{x_1,r}=p=\tilde
p_{x_0,\rho}+h_{x_0,\rho}$ in $B(x_1,r)$, by H\"older's inequality
\begin{multline}   \label{eq065.47}
\int_{B(x_1,r)}|[h_{x_0,\rho}]_{x_1,r}-[h_{x_1,r}]_{x_1,r}|^{3/2}\,dx
\\
=Nr^6|[h_{x_0,\rho}]_{x_1,r}-[h_{x_1,r}]_{x_1,r}|^{3/2}
= Nr^6|[\tilde p_{x_0,\rho}]_{x_1,r}-[\tilde
p_{x_1,r}]_{x_1,r}|^{3/2}
\\
\leq N\int_{B(x_1,r)}|{\tilde
p}_{x_0,\rho}|^{3/2}+|{\tilde p}_{x_1,r}|^{3/2}\,dx.
\end{multline}
From \eqref{eq065.42}, \eqref{eq065.47}, \eqref{eq28.1} and \eqref{eq065.23},
we get
\begin{multline*}
\int_{B(x_1,r)}|p-[h_{x_1,r}]_{x_1,r}|^{3/2}\,dx
\leq N\Big(\int_{B(x_0,\rho)}|\nabla u|^2\,dx\Big)^{3/2}
\\
+N\Big(\int_{B(x_0,\rho)}|f|^{6/5}\,dx\Big)^{5/4}
+N\left(\frac{r}{\rho}\right)^{15/2}
\int_{B(x_0,\rho)}|p-[h_{x_0,\rho}]_{x_0,\rho}|^{3/2}\,dx.
\end{multline*}
Finally, by H\"older's inequality, the lemma is proved.
\end{proof}

Note that the following estimates use the equation \eqref{NSeq1}, or more precisely they use the generalized local energy inequality \eqref{energy}.

\begin{lemma}
                    \label{lemma2.28.3}
Suppose $\theta\in (0,1/2]$ and $\rho>0$ are constants, and
$B(x_0,\rho)\subset \Omega$. Then we have
$$
A(\theta\rho)+E(\theta\rho)\leq N\theta^{-2}\big[
C^{2/3}(\rho)+C(\rho) +C^{1/3}(\rho)D^{2/3}(\rho)+F(\rho)\big].
$$
In particular, when $\theta=1/2$ we have
\begin{equation}
                    \label{eq11.36}
A(\rho/2)+E(\rho/2)
\leq N[C^{2/3}(\rho)+C(\rho)
+C^{1/3}(\rho)D^{2/3}(\rho)+F(\rho)].
\end{equation}
Here $N$ is a positive constant independent of $\theta$, $\rho$ and $x_0$.
\end{lemma}

\begin{proof}
Let $r=\theta \rho$. By H\"older's inequality,
$$
A(r)\le C^{2/3}(r)\le N\theta^{-2}C^{2/3}(\rho).
$$
To estimate $E(r)$, in the
energy inequality \eqref{energy} we choose a
suitable smooth cut-off function $\psi=\psi_1\in C_0^\infty(B(x_0,\rho))$ such that
$$
0\leq \psi_1\leq 1\,\,\text{in}\,B(x_0,\rho),\quad
\psi_1\equiv 1 \,\,\text{in}\,B(x_0,\rho/2)
$$
\begin{equation}
                                    \label{eq17.03}
|\nabla \psi_1| \le N\rho^{-1},\quad |\nabla^2 \psi_1|\le N\rho^{-2}\,\,\text{in}\,B(x_0,\rho).
\end{equation}
By using \eqref{energy} and because $u$ is divergence free, we get
\begin{multline*}
E(r) \leq \frac{N}{r^2}\Big[\frac{1}{\rho^2}
\int_{B(x_0,\rho)}|u|^2\,dx
\\
+\frac{1}{\rho}\int_{B(x_0,\rho)}(|u|^2+2|p-[h]_{x_0,\rho}|)|u|\,dx
+\int_{B(x_0,\rho)}|u||f|\,dx\Big].
\end{multline*}
Due to the H\"older inequality and Young's inequality, one obtains
$$
\int_{B(x_0,\rho)}|u|^2\,dx\leq
\big(\int_{B(x_0,\rho)}|u|^3\,dx\big)^{2/3}
\big(\int_{B(x_0,\rho)}\,dx\big)^{1/3}\leq \rho^4C^{2/3}(\rho).
$$
And
\begin{multline*}
\int_{B(x_0,\rho)}|p-[h]_{x_0,\rho}||u|\,dx
\\
\leq
\big(\int_{B(x_0,\rho)}|p-[h]_{x_0,\rho}|^{3/2}\,dx\big)^{2/3}
\big(\int_{B(x_0,\rho)}|u|^3\,dx\big)^{1/3}
\\
\leq N\rho^3D^{2/3}(\rho)C^{1/3}(\rho).
\end{multline*}
Furthermore
$$
\int_{B(x_0,\rho)}|u||f|\,dx\le \frac{1}{\rho^2}\int_{B(x_0,\rho)}|u|^2\,dx+
\rho^2\int_{B(x_0,\rho)}|f|^2\,dx.
$$
Then, collecting these estimates, Lemma \ref{lemma2.28.3} thus follows.
\end{proof}

As a conclusion of this subsection, we obtain

\begin{prop}\label{lemma05.03.02}
For any small $\epsilon_0>0$, there exists $\epsilon_1=\epsilon_1(\epsilon_0)>0$ small  such that
for any $x_0\in \Omega$ satisfying
\begin{equation} \label{eq3.02}
\limsup_{r\downarrow 0}E(r)\leq \epsilon_1,
\end{equation}
we have
\begin{equation}
                                    \label{eq3.02.2}
A(\rho_0)+E(\rho_0)+C(\rho_0)+D(\rho_0)\leq \epsilon_0,
\end{equation}
provided that $\rho_0$ is sufficiently small.
\end{prop}
\begin{proof}
For a given $x_0\in \Omega$ satisfying
\eqref{eq3.02}, choose $\rho_1>0$ such that $B(x_0,\rho_1)\subset
\Omega$. Then for any $\rho\in (0,\rho_1]$, by
using \eqref{eq11.36} and Young's inequality
$$
A(\gamma\rho)+E(\gamma\rho)\leq
N[C^{2/3}(2\gamma\rho)+C(2\gamma\rho)+D(2\gamma\rho)+F(2\gamma\rho)].
$$
This estimate, \eqref{eq11.35} and \eqref{eq11.30d}, with $\gamma\in (0,1/8)$, together with
Young's inequality again implies
\begin{multline} \label{eq17.11.15}
A(\gamma\rho)+E(\gamma\rho)+C(\gamma\rho)+D(\gamma\rho)
\\
\leq
N\left[\gamma^{2}C^{2/3}(\rho)+\gamma^{9/2} D(\rho)+\gamma^3
C(\rho)+\gamma^3 A(\rho)\right]
\\
+
N\gamma^{-50}\left[ E(\rho)+E^3(\rho)+F(\rho)\right] +N\gamma^{2}
\\
\leq
N\gamma^{2}\left[A(\rho)+E(\rho)+C(\rho)+D(\rho)\right]+N\gamma^{2}
\\
+
N\gamma^{-50}\left[ E(\rho)+E^3(\rho)+F(\rho)\right].
\end{multline}
Since $f\in L_{6,\text{loc}}$, by H\"older's inequality, we have
\begin{equation}\label{eq00.20}
F(\rho)\le \|f\|_{L_6(B(x_0,\rho_1))}^{2}\rho^4.
\end{equation}
It is easy to see that for any $\epsilon_0>0$, there are sufficiently
small real numbers $\gamma\leq 1/\sqrt{3N}$ and $\epsilon_1$ such
that if \eqref{eq3.02} holds then for all small $\rho$ we have
$$
N\gamma^{2}+N\gamma^{-50}(E(\rho)+E^3(\rho)+F(\rho))<\epsilon_0/2.
$$
By using \eqref{eq17.11.15}, we reach
$$
A(\gamma\rho)+E(\gamma\rho)+C(\gamma\rho)+D(\gamma\rho)\le \frac 1 3
\left[A(\rho)+E(\rho)+C(\rho)+D(\rho)\right]+\frac {\epsilon_0} 2,
$$
which together with a standard iteration argument gives \eqref{eq3.02.2}
for some $\rho_0>0$ small enough.
\end{proof}

\subsection{Step 2}
In the second step, first we will estimate the values of $A$, $E$ and
$D$ in a smaller ball by their values in a larger ball.  Note that in this subsection all of the quantities implicitly depend upon the point $x_1$ as $A(r)=A(r,x_1)$ unless it says so otherwise.

\begin{lemma}
                    \label{lemma05.18.1}
Fix constants $\rho>0$, $\theta\in (0,1/3]$ and
$B(x_1,\rho)\subset \Omega$. Then we have
\begin{multline}
                    \label{eq18.2.31}
A(\theta\rho)+E(\theta\rho)\leq N\theta^2
A(\rho)
\\
+N\theta^{-3}\big([A(\rho)+E(\rho)]^{3/2}+D(\rho)\big)
+N\theta^{-6}F(\rho),
\end{multline}
where $N>0$ is independent of $\rho$, $\theta$ and $x_1$.
\end{lemma}

\begin{proof}
We prove the lemma by using a suitably chosen test function in the generalized local energy inequality \eqref{energy}. Let $r=\theta\rho$. We define
$$
\psi_2(x)=(r^2+|x-x_1|^2)^{-2},
$$
which clearly satisfies $\Delta \psi_2 = - 24 r^2 (r^2+|x-x_1|^2)^{-4}$ so that
\begin{equation}
                        \label{eq16.40}
\Delta \psi_2  <0\quad\text{in}\quad \bR^6,\quad
\Delta \psi_2\le -cr^{-6}\quad\text{in}\quad B(x_1,\rho),
\end{equation}
for some constant $c>0$ independent of $r$.

In the energy inequality \eqref{energy} we choose
$\psi=\psi_1\psi_2$, where $\psi_1$ is taken
from \eqref{eq17.03} in the proof of Lemma \ref{lemma2.28.3} with the center $x_1$ in place of $x_0$.
Then we have
\begin{multline}
-\int_{B(x_1,\rho)}|u|^2\psi_1\Delta\psi_2\,dx
+2\int_{B(x_1,\rho)}|\nabla u|^2\psi_1\psi_2\,dx \\
                    \label{eq2.54}
\leq \int_{B(x_1,\rho)}\{|u|^2(\psi_2\Delta
\psi_1+2\nabla\psi_1\cdot\nabla \psi_2)\\
+(|u|^2+2(p-[h]_{x_1,\rho}))u\cdot (\psi_1\nabla\psi_2+\psi_2\nabla\psi_1)+f\cdot u\psi_1\psi_2\}\,dx.
\end{multline}
After some straightforward computations, from \eqref{eq17.03} and \eqref{eq16.40},
it is easy to see the following properties:
\begin{itemize}
\item[(i)] For some constant $c>0$, on $\bar B(x_1,r)$ it holds that
$$
\psi_1\psi_2=\psi_2\geq cr^{-4},\quad -\psi_1\Delta\psi_2=-\Delta\psi_2\ge cr^{-6}.
$$
\item[(ii)]  In  $B(x_1,\rho)$, we have
$$
|\psi_1\psi_2|\le Nr^{-4},\quad
|\psi_1\nabla\psi_2|+|\psi_2\nabla\psi_1|\leq Nr^{-5},
$$
$$
|\psi_2\Delta\psi_1|+|\nabla\psi_1\cdot\nabla\psi_2|
\leq N\rho^{-6}.
$$
\end{itemize}
These properties together with \eqref{eq2.54}, the Young and H{\"o}lder inequalities,
yield
\begin{equation}
                                    \label{eq4.47}
A(r)+E(r)\leq N[\theta^2 A(\rho) +\theta^{-3}(C(\rho)+D(\rho))+\theta^{-6} F(\rho)].
\end{equation}
Owing to the Sobolev inequality \eqref{eq11.30}, one easily gets
\begin{equation}
                                    \label{eq4.48}
C(\rho)\leq N[A(\rho)+E(\rho)]^{3/2}.
\end{equation}
Upon combining \eqref{eq4.47} and
\eqref{eq4.48}, the lemma is proved.
\end{proof}

\begin{lemma}\label{lemma05.3.1}
Suppose $\rho>0$ is constant and $B(x_1,\rho)\subset \Omega$. Then we
can find a $\theta_1\in (0,1)$  small, where $\theta_1$ does not depend upon $\rho$,  such that
\begin{multline}\label{eq5.3.1}
A(\theta_1\rho)+E(\theta_1\rho)+D^{2/3}(\theta_1\rho)\leq
\frac{1}{4}\big[A(\rho)+E(\rho)+D^{2/3}(\rho)\big]
\\
+N(\theta_1)\big[A(\rho)+E(\rho)+D^{2/3}(\rho)\big]^{3/2}
+N(\theta_1)\big[F(\rho)+F^{1/2}(\rho)\big],
\end{multline}
where $N$ is a constant independent of $\rho$ and $x_1$.
\end{lemma}

\begin{proof}
Due to \eqref{eq11.30d} and \eqref{eq18.2.31},
for any $\gamma,\theta\in (0,1/4]$, we have
$$
D^{2/3}(\gamma\theta\rho)\leq
N\big[\gamma^{3}D^{2/3}(\theta\rho)
+\gamma^{-2}F^{1/2}(\theta\rho)
+\gamma^{-2}E(\theta\rho)
\big]
$$
$$
\leq N\gamma^{3}\theta^{-2}D^{2/3}(\rho)+N\gamma^{-2}F^{1/2}(\rho)
+N\gamma^{-2}\theta^2 A(\rho)
$$
\begin{equation}
                                \label{eq6.18}
+N\gamma^{-2}\theta^{-3}
\big[A(\rho)+E(\rho)+D^{2/3}(\rho)\big]^{3/2} +N\gamma^{-2}\theta^{-6}F(\rho),
\end{equation}
and from \eqref{eq18.2.31} we have
\begin{multline}
                    \label{eq18.6.25}
A(\gamma\theta\rho)+E(\gamma\theta\rho)
\leq N(\gamma\theta)^2
A(\rho)
\\
+N(\gamma\theta)^{-3}[A(\rho)+E(\rho)+D^{2/3}(\rho)]^{3/2}
+N(\gamma\theta)^{-6} F(\rho).
\end{multline}
Now we choose and fix $\theta$ sufficiently small and $\gamma=\theta^{4/5}$
such that
$$
N[\gamma^3\theta^{-2}+\gamma^{-2}\theta^2+(\gamma\theta)^2]\le N\theta^{2/5}\leq
1/8.
$$
Upon adding \eqref{eq6.18} and \eqref{eq18.6.25}, we obtain
\begin{multline*}
A(\gamma\theta\rho)+E(\gamma\theta\rho)+D^{2/3}(\gamma\theta\rho)
\\
\leq
\frac{1}{4}\left[ A(\rho)+D^{2/3}(\rho) \right]
+
N[A(\rho)+E(\rho)+D^{2/3}(\rho)]^{3/2}
\\
+N\big[F(\rho)+F^{1/2}(\rho)\big],
\end{multline*}
where $N$ depends only on $\theta$ and $\gamma$. After putting
$\theta_1=\gamma\theta$, the lemma is proved.
\end{proof}

In the next proposition we
will study 
the decay property of $A$, $E$, $C$ and $D$ as the radius $\rho$ goes to zero.

\begin{prop}\label{lemma05.03.18.6}
There exists $\epsilon_0>0$ satisfying the following property.
Suppose that for some $x_0\in \Omega$ and $\rho_0\in (0,1)$
satisfying $B(x_0,\rho_0)\subset \Omega$ we have
\begin{equation}\label{eq3.02.2.b.b}
C(\rho_0,x_0)+D(\rho_0,x_0)+F(\rho_0,x_0)\leq \epsilon_0.
\end{equation}
Then we can find $N>0$ and $\alpha_0\in (0,1)$ such that for any $\rho\in (0,\rho_0/8)$ and $x_1\in B(x_0,\rho_0/8)$, the following inequality will hold uniformly
\begin{equation}\label{eq18.6.48}
A(\rho,x_1)+E(\rho,x_1)
+C^{2/3}(\rho,x_1)
+D^{2/3}(\rho,x_1)\leq N\rho^{\alpha_0},
\end{equation}
where $N$ is a positive constant independent of $\rho$ and $x_1$.
\end{prop}

\begin{proof}
Fix the constant $\theta_1$ from Lemma \ref{lemma05.3.1}. Due to
\eqref{eq11.36},
 \eqref{eq11.30d}
 and \eqref{eq3.02.2.b.b}, we may first choose $\epsilon'>0$ then $\epsilon_0=\epsilon_0(\epsilon')>0$ sufficiently small such that, for any $x_1\in B(x_0,\rho_0/8)$,
$$
A(\rho_0/4,x_0)+E(\rho_0/4,x_0)+D^{2/3}(\rho_0/8,x_1)
\leq \frac{\epsilon'}{16},
$$
and
\begin{equation}\label{eq9.58}
N(\theta_1)\sqrt{\epsilon'} \leq 1/4,\quad N(\theta_1)(\epsilon_0+\epsilon_0^{1/2})\le \epsilon'/2.
\end{equation}
where  $N(\theta_1)>0$ is the same constant from
\eqref{eq5.3.1}. By using
$$
B(x_1,\rho_0/8)\subset B(x_0,\rho_0/4)\subset \Omega,
$$
we then have
$$
\varphi(\rho_0)
:=
A(\rho_0/8,x_1)+E(\rho_0/8,x_1)+D^{2/3}(\rho_0/8,x_1)\leq \epsilon'.
$$
By using \eqref{eq9.58} and \eqref{eq5.3.1} with $\rho = \rho_0/8$ we obtain inductively that
$$
\varphi(\theta_1^k \rho_0)
=
A(\theta_1^k\rho_0/8,x_1)+E(\theta_1^k\rho_0/8,x_1)
+D^{2/3}(\theta_1^k\rho_0/8, x_1)
\le \epsilon'.
$$
(Holding for $k=1, 2, \ldots$).
It then similarly follows from \eqref{eq9.58} and \eqref{eq5.3.1} that
\begin{equation}
\label{to.iterate}
\varphi(\theta_1^k \rho_0)
\le \frac{1}{2} \varphi(\theta_1^{k-1} \rho_0)
+N_1 (\theta_1^{k-1} \rho_0)^2.
\end{equation}
Above, thanks to \eqref{eq00.20}, we have used the estimate
$$
F(\theta_1^{k-1}\rho_0/8,x_1)+F^{1/2}(\theta_1^{k-1}\rho_0/8,x_1)
\le
N_1(\|f\|_{L_6(B(x_0,\rho_0/2))}) ~  (\theta_1^{k-1} \rho_0)^2.
$$
Now we use a standard iteration argument to obtain the H{\"o}lder continuity of $\varphi$.   We have to be a bit careful however because we do not make the standard assumption that $\varphi(\rho)$ should be a non-decreasing function.  We iterate \eqref{to.iterate} to obtain
\begin{multline}            \label{eq21.3.10}
\varphi(\theta_1^k \rho_0)
\le
\left(\frac{1}{2}\right)^k \varphi( \rho_0)
+
N_1 \rho_0^2 \sum_{j=0}^{k-1} \left(\frac{1}{2}\right)^j (\theta_1^{k-1-j} )^2
\\
\le
\left(\frac{1}{2}\right)^k \left[\varphi( \rho_0)
+
\frac{2N_1}{1-\theta_1} \rho_0^2 \right].
\end{multline}
In the last inequality, without loss of generality we have used that $\theta_1\in(0,1/2]$.
Since $\rho \in (0, \rho_0/32)$ we can find $k$ such that
$
\theta_1^{k} \frac{\rho_0}{8}
<
4\rho
\le
\theta_1^{k-1} \frac{\rho_0}{8}.
$
Then
$$
A(\rho,x_1)+E(\rho,x_1)+D^{2/3}(\rho,x_1)
\le
N(\theta_1)\big(\varphi(\theta_1^{k-1} \rho_0)+F^{1/2}(\theta_1^{k-1} \rho_0,x_1)\big),
$$
where we used Lemma \ref{lemma2.28.2} to estimate the third term on the left-hand side.
By \eqref{eq21.3.10} and \eqref{eq00.20},
the above is further bounded by
$$
N(\theta_1)
\left(\frac{1}{2}\right)^k \left[\varphi( \rho_0)
+
\frac{2N_1}{1-\theta_1} \rho_0^2 \right]+N(\theta_1)\rho^2
\le
N \rho^{\alpha_0}.
$$
In this last line $N=N(\theta_1, \varphi( \rho_0), N_1, \rho_0)$ and $\alpha_0 = \frac{\log(1/2)}{\log(\theta_1)}>0$.
This yields \eqref{eq18.6.48} for the terms $A$, $E$ and $D$.
The inequality for $C(\rho,x_1)$ follows from \eqref{eq4.48}.
\end{proof}

\subsection{Step 3 -- Proofs of Theorems \ref{thm3}-\ref{thm1}}  \label{Sec5}
In the final step, we are going to use a bootstrap argument  to successively improve the decay estimate \eqref{eq18.6.48}. However, as we will show below, the bootstrap argument itself only gives the decay of $E(\rho)$ no more than $\rho^2$,
i.e. one can obtain an estimate like
$$
\int_{B(x_1,\rho)}|\nabla u|^{2}\,dx\le N(\epsilon) \rho^{4-\epsilon}, \quad \forall \epsilon>0,
$$
for any $\rho$ sufficiently small. Unfortunately, this decay estimate is not enough for the H\"older regularity of $u$ since the dimension is six (so that we need the exponent $4+\epsilon$ instead of $4-\epsilon$ according to the Morrey lemma).
Then to fill in this gap we will use the elliptic theory.

First we prove Theorem \ref{thm4}. We begin with the bootstrap argument.  We will choose an increasing sequence of real numbers $\{\alK\}_{k=1}^\infty \in (\alpha_0,2)$ such that for any small $\delta>0$ we can find an integer $m=m(\delta)$ with the property that $\alpha_m>2-\delta$.

For a fixed $\delta>0$ and $m=m(\delta)$, under the condition \eqref{eq4.1.11}, we {\it claim} that the following estimates hold uniformly for all $\rho >0$ sufficiently small and $x_1\in B(x_0,\rho_0/8)$ over the range of $\{\alK\}_{k=0}^m$:
\begin{multline}\label{eq10.53}
A(\rho,x_1)+E(\rho,x_1)\le N\rho^{\alK}, \quad
\\
C(\rho,x_1)\le N\rho^{3\alK/2},\quad
D(\rho,x_1)\le N\rho^{3\alK/2}.
\end{multline}
We prove this via iteration.  The  $k=0$ case for \eqref{eq10.53} with $\alpha_0$ was proven in \eqref{eq18.6.48}.

We first estimate $A(\rho,x_1)$ and $E(\rho,x_1)$.  Let $\rho = \tilde{\theta} \tilde{\rho}$ where $\tilde{\theta} =  \rho^\mu$,
$\tilde{\rho} = \rho^{1-\mu}$ and $\mu \in (0,1)$ to be determined.  We use Lemma \ref{lemma05.18.1} and then \eqref{eq10.53} (for $\alK$) to obtain
$$
A(\rho) + E(\rho)
\le
N
\left(
\rho^{2\mu+\alK(1-\mu)}
+
\rho^{\frac{3}{2}\alK(1-\mu) - 3 \mu}
+
\rho^{4(1-\mu)-6 \mu}
\right).
$$
Choose $\mu = \frac{\alK}{10+\alK} $,
then \eqref{eq10.53} is proven for $A(\rho) + E(\rho)$ with the exponent of
\begin{align*}
\alpha_{k+1}:=
\min\Big\{2\mu+\alK(1-\mu),\frac{3}{2}\alK(1-\mu) - 3 \mu, 4(1-\mu)-6 \mu \Big\}\\
=\frac{3}{2}\alK(1-\mu) - 3 \mu=\frac{12}{10+\alK} \alK\in (\alK,2).
\end{align*}
Then the estimate in \eqref{eq10.53} (with $\alpha_{k+1}$) for $C(\rho,x_1)$ follows from \eqref{eq4.48}.  To prove the estimate in  \eqref{eq10.53} (after level $k$) for $D(\rho,x_1)$ we will use Lemma \ref{lemma2.28.2}.
From \eqref{eq11.30d} we obtain
$$
D(\gamma\rho,x_1)
\le
N
\left(
\gamma^{9/2}D(\rho,x_1)+\gamma^{-3}\rho^{3\alpha_{k+1}/2}
+
\gamma^{-3}\rho^3
\right).
$$
The estimate used here for $F(\rho)$ follows from \eqref{eq00.20}. Now for any $r$ small, we take the supremum on both sides with respect to $\rho\in (0,r)$ and get
$$
\sup_{\rho\in (0,r]}D(\gamma\rho,x_1)
\le
N
\gamma^{9/2}\sup_{\rho\in (0,r]}D(\rho,x_1)+N\gamma^{-3}r^{3\alpha_{k+1}/2}
+N\gamma^{-3}r^3.
$$
Since $9/2>3>\frac{3}{2}\alpha_{k+1}$, by using a well-known iteration argument, similar to \eqref{to.iterate} (or see e.g., \cite[Chap. 3, Lemma 2.1]{Gi83}), we obtain the estimate in \eqref{eq10.53} (with $\alpha_{k+1}$) for $D(\rho)$.  Then we have shown how to build the increasing sequence of $\{\alK\}$ for which \eqref{eq10.53} holds. Moreover,
$$
2-\alpha_{k+1}=\frac {10} {10+\alpha_k}(2-\alpha_k)\le \frac {10} {10+\alpha_0}(2-\alpha_k),
$$
which implies that $\alpha_k\to 2$ as $k\to \infty$.
Note that by the above proof, the constant $N$ in \eqref{eq10.53} may go to infinity as $k\to\infty$; thus we truncate at level $m<\infty$.

In particular, \eqref{eq10.53} with $k=m$ gives for any small $\delta = \delta(m)>0$ that
\begin{align}
                                \label{eq10.54}
\int_{B(x_1,\rho)}|u|^2\,dx&\le N\rho^{6-\delta},\\
                                \label{eq10.59}
\int_{B(x_1,\rho)}|u|^3+|p-[h]_{x_1,\rho}|^{3/2}\,dx&\le N\rho^{6-\frac{3}{2}\delta}.
\end{align}
We obtained these estimates via the bootstrap argument, next we will use the elliptic theory to improve them.

Now we fix a $\delta\in (0,1/10)$ and rewrite \eqref{NSeq1}  (in the sense of distributions) into
$$
\Delta u_i=D_j(u_iu_j)+D_i p-f_i.
$$
Finally, we use the classical elliptic theory to complete the proof.
Thanks to \eqref{eq10.54}, there exists $\rho_1\in (\rho/2,\rho)$ such that
\begin{equation}
                                        \label{eq10.12.07}
\int_{S(x_1,\rho_1)}|u|^2\,dx \le N\rho^{5-\delta}.
\end{equation}
Let $v$ be the unique $H^1$ solution to the Laplace equation
$$
\Delta v_i=0 \quad \text{in}\,\,B(x_1,\rho_1),
$$
with the boundary condition $v_i=u_i$ on $S(x_1,\rho_1)$. It follows from the standard estimates for harmonic functions, H\"older's inequality, and \eqref{eq10.12.07} that
\begin{equation}
                                    \label{eq10.16.45}
\sup_{B(x_1,\rho_1/2)}|\nabla v|\le N\rho_1^{-6}\int_{S(x_1,\rho_1)}|v|\,dx\le N\rho^{-1-\delta/2}.
\end{equation}
Denote $w=u-v\in H^1(B(x_1,\rho_1))$. Then $w$ satisfies the Poisson equation	
$$
\Delta w_i=D_j(u_iu_j)+D_i (p-[h]_{x_1,\rho})-f_i \quad \text{in}\,\,B(x_1,\rho_1).
$$
with zero boundary condition on $S(x_1,\rho_1)$. By the classical $L_p$ estimates for the Poisson equation, we have
\begin{multline*}
\|\nabla w\|_{L_{3/2}(B(x_1,\rho_1))}\le
N\left\||u|^2\right\|_{L_{3/2}(B(x_1,\rho_1))}\\
+N\left\|p-[h]_{x_1,\rho}\right\|_{L_{3/2}(B(x_1,\rho_1))}
+N\rho_1\left\|f\right\|_{L_{3/2}(B(x_1,\rho_1))}.
\end{multline*}
This together with the assumption on $f$ and \eqref{eq10.59} gives
\begin{equation}
                                    \label{eq10.16.46}
\|\nabla w\|_{L_{3/2}(B(x_1,\rho_1))}\le
N\rho^{4-\delta}+N\rho^{4} \le
N\rho^{4-\delta}.
\end{equation}
Since $|\nabla u|\le |\nabla w|+|\nabla v|$, we combine \eqref{eq10.16.45} and \eqref{eq10.16.46} to obtain, for any $r\in (0,\rho/4)$, that
$$
\int_{B(x_1,r)}|\nabla u|^{3/2}\,dx\le N\rho^{6-3\delta/2}+Nr^6 \rho^{-3/2-3\delta/4}.
$$
Upon taking $r=\rho^{5/4-\delta/8}/4$ (with $\rho$ small), we get
\begin{equation}
                                    \label{eq17.31}
\int_{B(x_1,r)}|\nabla u|^{3/2}\,dx\le N r^{\beta},
\end{equation}
where
$$
\beta=\frac{6-3\delta/2}{5/4-\delta/8}>6-3/2.
$$
Since \eqref{eq17.31} holds for arbitrary $x_1\in B(x_0,\rho_0/8)$ and all $r$ small, by the Morrey lemma (see for instance 
\cite[Theorem 1.1 on p. 64 of Ch. III]{Gi83}),
$u$ is H\"older continuous in a neighborhood of $x_0$. This completes the proof of Theorem \ref{thm4}.

Theorem \ref{thm3} then follows from Theorem \ref{thm4} by applying Proposition \ref{lemma05.03.02}.
Finally, Theorem \ref{thm1} is deduced from Theorem \ref{thm3} by using the standard argument in the geometric measure theory, which is explained for example in \cite{CKN},  or alternatively in \cite{Gi83}.

\begin{remark}
                                    \label{finalrem}
Finally we remark that by using the same method we can get an alternative proof of Theorem \ref{thm0} for the 5D steady-state Navier-Stokes equations if we assume that $f\in L_{5,\text{loc}}$. However, it seems to us that six is the highest
dimension to which our approach (or any existing approach) applies.
In fact, by the Sobolev imbedding theorem, $H^1(\bR^6)\hookrightarrow L_3(\bR^6)$. So the nonlinear term in the energy inequality can be controlled by the energy norm when $d=6$ but not higher.
\end{remark}


\end{document}